\documentclass[11pt]{amsart}

\usepackage{amssymb,amsmath,enumerate,epsfig}
\usepackage{amsfonts,color}
\usepackage{a4,latexsym}
\usepackage{parskip}
\usepackage{hyperref}
\hypersetup{pdftitle=Segre4}

\newcommand{\C} {\mathbb{C}}
\newcommand{\Q} {\mathbb{Q}}

\newcommand{\F}{\mathbb{F}}
\newcommand{\Z}{\mathbb{Z}}

\newcommand{\OO}{\mathcal{O}}

\newcommand{\PP}{\mathbb{P}}
\newcommand{\NS}{\mathop{\rm NS}}

\newcommand{\MW}{\mathop{\rm MW}}

\newcommand{\disc}{\mathop{\rm disc}}
\newcommand{\A}{\mathbb{A}}

\newcommand{\Pic}{\mathop{\rm Pic}}

\newcommand{\mF}{\mathcal F}

\newtheorem{Theorem}{Theorem}[section]
\newtheorem{Proposition}[Theorem]{Proposition}
\newtheorem{Lemma}[Theorem]{Lemma}

\theoremstyle{remark}
\newtheorem{Fact}[Theorem]{Fact}
\newtheorem{Remark}[Theorem]{Remark}

\newtheorem{Example}[Theorem]{Example}

\theoremstyle{definition}

\begin{document}

\title{112 lines on smooth quartic surfaces (characteristic 3)}


\author{S\l awomir Rams}
\address{Institute of Mathematics, Jagiellonian University, 
ul. {\L}ojasiewicza 6,  30-348 Krak\'ow, Poland}
\address{Current address: 
Institut f\"ur Algebraische Geometrie, Leibniz Universit\"at
  Hannover, Welfengarten 1, 30167 Hannover, Germany} 
\email{slawomir.rams@uj.edu.pl}

\author{Matthias Sch\"utt}
\address{Institut f\"ur Algebraische Geometrie, Leibniz Universit\"at
  Hannover, Welfengarten 1, 30167 Hannover, Germany}
\email{schuett@math.uni-hannover.de}

\thanks{Funding   by ERC StG~279723 (SURFARI) and  NCN grant N N201 608040 (S.~Rams)
 is gratefully acknowledged.}

\date{May 6, 2015}

\begin{abstract}
Over a field $k$ of characteristic $3$,
 we prove 
 that there are no smooth quartic surfaces $S \subset \PP^3_k$ 
 with more than 112 lines.
 Moreover, the surface with 112 lines is projectively equivalent over $\bar k$ to the Fermat quartic.
 As a key ingredient, we derive a characteristic free upper bound
 for the number of lines met by a conic on a smooth quartic surface.
\end{abstract}
%
%
 \maketitle

 \section{Introduction}
 \label{s:intro}
 
 The study of lines on surfaces in $\PP^3$ is a classical topic in algebraic geometry,
 starting from the prototype 27 lines on a smooth cubic.
For higher degree, however, 
smooth surfaces do generically not contain any lines at all,
so one can only aim for upper bounds on their number.
In \cite{RS} we derived the sharp upper bound of 64 lines
for smooth quartics over fields of characteristic other than $2,3$,
correcting Segre's arguments \cite{Segre} over $\C$.
Notably, our proof used elliptic fibrations and the flecnodal divisor (see Section \ref{s:Hess})
both of which may degenerate in characteristics $2, 3$.
The aim of this paper is to cover these two remaining characteristics.
To this end, we derive the following characteristic-free bound
for conics on smooth quartic surfaces
which will serve as a replacement of the flecnodal divisor:

\begin{Theorem}
\label{thm:Q}
A geometrically irreducible conic on a smooth quartic $S\subset\PP^3$
may meet at most 48 lines on $S$.
\end{Theorem}

The proof of Theorem \ref{thm:Q} will be given in Section \ref{s:Hess} using the Hessian quadric.
We should like to note that similar techniques apply to surfaces of higher degree
as we will exploit in future work.
Presently, we will concentrate on the consequences for characteristics $2$ and $3$.
Our main theorem is as follows:

\begin{Theorem}
\label{thm}
Let $k$ be a field of characteristic $3$.
Then any smooth quartic surface over $k$ contains at most 112 lines.
\end{Theorem}

We will prove Theorem \ref{thm} in Section \ref{s:pf}.
Let us point out that Theorem \ref{thm} is actually sharp as the bound is attained by the Fermat quartic
(see Example \ref{ex:Fermat}).
Section \ref{s:unique} details  a proof that the Fermat surface is indeed uniquely determined 
by the number of lines on it up to projective transformations over $\bar k$
(compare the uniqueness over $\F_9$ in \cite{HS}).
The paper concludes with a sketch of the proof of the corresponding bound for characteristic $2$:

\begin{Proposition}
\label{prop:2}
Let $k$ be a field of characteristic $2$.
Then any smooth quartic surface over $k$ contains at most 84 lines.
\end{Proposition}

In contrast to the case of characteristic 3, the latter bound is not known to be sharp.
We hope to remedy this discrepancy in future work.

\section{Hessian quadric}
\label{s:Hess}

The statements of the introduction remain valid if one replaces $k$ by a field extension,
hence we will assume throughout this paper that $k$ is algebraically closed.
Then we fix a smooth quartic surface
\[
S=\{f=0\}\subset\PP^3_k\]
given by a homogeneous irreducible polynomial
$f\in k[x_1,x_2,x_3,x_4]$ of degree $4$.
In this section, we give a proof of Theorem \ref{thm:Q}.
As a motivation, we briefly sketch the meaning of the flecnodal divisor 
as hinted in the introduction.
The flecnodal divisor $\mF$ was a key ingredient of \cite{Segre}, \cite{RS}
and roughly contains all the points on $S$ which admit a 4-contact line in $\PP^3$.
As a scheme,
it is defined as the intersection
\[
\mF = S \cap Y
\]
where $Y\subset\PP^3$ is the total space of the 4-contact lines to $S$ in $\PP^3$:
\[
Y = \cup \{\ell\subset\PP^3 \text{ a line}; \; \ell\subset S \; \text{ or } \; \exists P\in S: \ell.S = 4P\}
\]
It goes back to Clebsch \cite{Clebsch} and Salmon  \cite{Salmon} that
\[
\mF\in H^0(S,\OO_S(20)) \;\;\; \text{ for } k=\C,
\]
but  by \cite[Thm.~1 \& Prop.~1]{Voloch} the same holds outside characteristics $2,3$.
It follows right away that $S$ contains at most 80 lines unless $\mF$ is degenerate
(i.e.~$\mF$ contains all of $S$),
and that an irreducible conic $Q\subset S$ meets at most $40=Q.\mF$ lines on $S$
unless $\mF$ is supported on $Q$.
In the latter case, however, we are not aware of a direct argument 
improving on Theorem \ref{thm:Q},
so although it is mostly aimed towards characteristics $2$ and $3$,
it may also have some independent relevance.

For later reference, we briefly comment on the degenerate case.
There are many equivalent formulations of this; for instance the Gauss map to the dual variety $S^*$
fails to be birational (as is generically the case).
Here is the prototype example in characteristic $3$:

\begin{Example}[Fermat quartic]
\label{ex:Fermat}
Let $k=\bar\F_3$ and consider the Fermat quartic
\[
S_3 = \{x_1^4+x_2^4+x_3^4+x_4^4=0\}\subset\PP^3_k.
\]
Then the Gauss map is purely inseparable,
and following the above discussion the flecnodal divisor degenerates.
Indeed $S_3$ contains 112 lines (each defined over $\F_9$, compare \cite{HS}), see \cite[p.~44]{Segre} or  \cite[\S 6.1]{SSvL}, e.g.,
and hence $\mF$ also has to degenerate from this point of view.
\end{Example}

Given a line $\ell\subset S$ and a point $P\in\ell$,
it is clear that $\ell$ is contained in the tangent plane to $S$ at $P$:
\begin{eqnarray}
\label{eq:T_P}
\ell\subset T_P.
\end{eqnarray}
We will compare this with the Hessian quadric $V_P\subset\PP^3$ defined as follows.
Let   $q_P\in k[x_1,x_2,x_3,x_4]$ denote the homogeneous quadratic form
defined by the Hessian at $P$ in terms of formal derivatives:
\[
q_P(x_1,x_2,x_3,x_4) =  \frac 12\, (x_1,x_2,x_3,x_4) \left(  \frac{\partial^2 f}{\partial x_i \partial x_j}(P)\right)_{1\leq i,j\leq 4} \phantom{l}^t(x_1,x_2,x_3,x_4).
\]
(This gives rise to well-defined quadratic form in characteristic $2$ by spelling out $q_P$ in terms of the coefficients of $f$ over $\Z$.) Then
\[
V_P = \{q_P=0\}\subset\PP^3.
\]
A local computation reveals that a line $\ell\subset\PP^3$ is contained in $V_P$
if $\ell$ meets $S$ with contact at least 3 at $P$.
In particular, this shows that the intersection
$
T_P \cap V_P
$
generically consists of the two 3-contact lines of $S$ at $P$,
and if $\ell\subset S$ contains $P$,
then
\begin{eqnarray}
\label{eq:V_P}
\ell\subset V_P.
\end{eqnarray}
We shall now specialise to the situation where $S$ contains an irreducible conic $Q$.
Then at any point $P\in Q$, we can form the intersection $T_P\cap V_P$.
We are interested in the total space
\[
Z = Z_Q = \cup_{P\in Q} (T_P\cap V_P) \subset \PP^3.
\]

\begin{Proposition}
\label{prop:Z}
$Z$ is a surface of degree $16$ in $\PP^3$.
It does not contain $S$,
but the surface $Z$ contains
any line in $S$ meeting $Q$.
\end{Proposition}

\begin{proof}
We start by proving that $Z$ is indeed a surface in $\PP^3$.
To this end, we first have to rule out that $V_P$ degenerates to $\PP^3$ for all $P\in Q$.
After a projective transformation,
we can assume that
\[
Q = \{x_4=x_1x_2-x_3^2=0\}\subset\PP^3.
\]
Consider the affine parametrisation 
\begin{eqnarray*}
\A^1 & \to & \;\;\;\;\;\;\;\;Q\\
t\;\; & \mapsto & P(t) =  [t^2,1,t,0].
\end{eqnarray*}
(We chose the affine parametrisation to simplify our notation.
In what follows, all degrees in $t$ should, however, be understood 
in terms of the homogeneous coordinates of $\PP^1$ parametrising $Q$;
this ensures that there are no degeneracies.)
Writing out
\[
f = \sum_{i,j,m} a_{i,j,m} x_1^ix_2^jx_3^mx_4^{4-i-j-m},\;\;\; a_{i,j,m}\in k,
\]
we can express the coefficients of the quadratic form $q_P$ over $k[t]$ as $P$ ranges over $Q$.
For instance, the coefficient of $x_1^2$ reads
\begin{eqnarray}
\label{eq:coeff}
6a_{4,0,0}t^4+3(a_{3,0,1}t^3+a_{3,1,0}t^2) + a_{2,0,2}t^2+a_{2,1,1}t+a_{2,2,0}.
\end{eqnarray}
If  $V_P$ were to degenerate for all $P\in Q$,
then each of these coefficients of $q_P$ would vanish identically over $k[t]$.
Thus we obtain a plentitude of linear equations in the coefficients of $f$.
If char$(k)\neq 2,3$, then the simultaneous solution directly leads to the conclusion $x_4\mid f$,
giving the required contradiction. 
In characteristics $2,3$, there are fewer equations,
but they still suffice.
For instance, if char$(k)=3$, then $f$ may in the end comprise all monomials
with some coordinate at least cubed and a multiple of $x_4^2(x_1x_2-x_3^2)$.
On the other hand, we know by assumption that
\[
f = x_4\cdot (\text{cubic}) + (x_1x_2-x_3^2)\cdot(\text{quadratic in } x_1, x_2, x_3)
\]
which cannot be matched with the shape prescribed by the indicated monomials.
The analogous argument holds in characteristic $2$
where the multiple of $x_4^2(x_1x_2-x_3^2)$ is supplemented by $x_1^4, x_2^4, x_3^4$ and multiples of $x_4^3$.

Secondly we have to check that generically on $Q$
\begin{eqnarray}
\label{eq:TV}
T_P \not\subset V_P.
\end{eqnarray}
To this end,
we use the linear polynomial $h\in k[t][x_1,\hdots,x_4]$ defining $T_P$ as $P$ ranges over $Q$.
Since $\deg_t(h)=6$ while $\deg_t(q_{P(t)})=4$,
we could only have a generic inclusion $T_P\subset V_P$ if $h$ where to split off a factor in $t$.
But then $T_P$ would degenerate at the roots of this factor,
i.e. $S$ would be singular at $P$, contradiction.

We have seen that the intersection $T_P\cap V_P$ has dimension one except at a finite number of points on $Q$.
Hence the total space $Z$ will be a surface in $\PP^3$.
We can compute a defining polynomial of $Z$ as the resultant $g\in k[x_1,x_2,x_3,x_4]$
of $q_{P(t)}$ and $h$ with respect to $t$.
Using the representation through the determinant of the Sylvester matrix,
we compute the degree of $g$
\[
\deg(g) = \deg_t(q_{P(t)}) \deg_{\underline x}(h) +  
\deg_t(h) \deg_{\underline x}(q_{P(t)}) = 4\cdot 1 + 6\cdot 2 = 16.
\]
Hence $Z$ is given by a polynomial of degree $16$ as claimed.
It may contain different irreducible components,
for instance accounting for the points on $Q$ where \eqref{eq:TV} degenerates,
but by construction each component is covered by a one-dimensional union of lines.
This rules out $S$ as a component of $Z$.

By \eqref{eq:T_P} and \eqref{eq:V_P},
the surface $Z$ contains any line $\ell$ meeting $Q$,
confirming the final claim of Proposition \ref{prop:Z}.
\end{proof}

\begin{Remark}
Similar arguments apply to surfaces in $\PP^3$ of higher degree,
notably to quintics,
as we shall exploit in future work.
\end{Remark}

\subsection*{Proof of Theorem \ref{thm:Q}}

Consider a conic $Q$ on a smooth quartic surface $S\subset\PP^3$.
Construct the total space of lines $Z\subset\PP^3$ as above and regard it as  
a divisor $D$ on $S$, i.e.~in $\OO_S(16)$ by Proposition \ref{prop:Z}.
Let $m>0$ denote the multiplicity of $Q$ in $Z$, i.e. the unique integer such that
$D-mQ$ is effective, but
\[
Q\not\subset \text{supp}(D-mQ).
\]
Since any line in $S$ meeting $Q$ is contained in the support of $D-mQ$ by Proposition \ref{prop:Z},
we obtain the following two upper bounds for their total number:
\begin{eqnarray}
\label{eq:48}
\#\{\text{lines on $S$ meeting } Q\} \leq
\begin{cases}
\deg(D-mQ) = 64-2m\\
Q.(D-mQ) = 32+2m
\end{cases}
\end{eqnarray}
Comparing the two bounds directly gives Theorem \ref{thm:Q}. \qed

\section{112 lines}
\label{s:pf}

In this section we specialise to the case where $k$ is algebraically closed of characteristic $3$.
This assumption makes a big difference for the fibration
\[
\pi: S\to \PP^1
\]
of genus one curves induced by a line $\ell$ on a smooth quartic surface $S\subset\PP^3$
by way of the pencil of hyperplanes containing $\ell$.
Notably, this fibration may become quasi-elliptic,
i.e. the generic fibre may be a cuspidal cubic curve.
This degeneration causes the upper bound for the number of lines incident with $\ell$ to go up from
20 (the bound from \cite{RS} correcting Segre's bound of 18 from \cite{Segre}):

\begin{Lemma}
\label{lem:30}
Depending on the induced fibration $\pi$,
a line on $S$ meets at most
\[
\begin{cases}
24 \text{ lines on $S$} & \text{ if $\pi$ is elliptic},\\
30 \text{ lines on $S$} & \text{ if $\pi$ is quasi-elliptic}.
\end{cases}
\]
\end{Lemma}

\begin{proof}
We use the Euler-Poincar\'e characteristic
expressed as sum over the contributions from the fibres of $\pi$.
Any line $\ell'\subset S$ meeting $\ell$ appears as component of a reducible singular fibre;
since $S$ is smooth, the possible fibre types are quite limited.
The following table collects the fibre types in Kodaira's notations 
and the local contribution to the Euler-Poincar\'e characteristic:

\begin{table}[ht!]
\begin{tabular}{c|l|c}
fibre type 
& configuration of fibre components &  $e(F_v)$\\
\hline
$I_1$ & nodal cubic & $1$\\
$I_2$ & line and conic meeting transversally in 2 points & $2$\\
$I_3$ & 3 lines meeting  in 3 points & $3$\\
$II$ & cuspidal cubic & $2+\delta$\\
$III$ & line and conic meeting tangentially in a point & $3$\\
$IV$ & 3 lines meeting in a single point & $4+\delta$
\end{tabular}
\vspace{.3cm}
\caption{Singular fibres of genus one pencils on smooth quartic surfaces}
\label{tab:F}
\end{table}
Here $\delta$ accounts for the wild ramification (which is zero on a quasi-elliptic fibration and positive 
on an elliptic fibration, see \cite{SS}).

If $\pi$ is an elliptic fibration,
then we have \begin{eqnarray}
\label{eq:ee}
24 = e(S) =   \sum_v e(F_v).
\end{eqnarray}
Comparing with the Table \ref{tab:F},
we find that this accommodates at most 24 lines,
coming in triangles forming $I_3$ fibres.
This gives the first claim of the lemma.

If $\pi$ is a quasi-elliptic fibration,
the general fibre $F$ has $e(F)=2$ instead of zero.
Thus we have
\begin{eqnarray}
\label{eq:eqe}
\;\;\;\;\;
24 = e(S) = e(\PP^1)\cdot e(F) +   \sum_v (e(F_v)-e(F)) = 4 +  \sum_v (e(F_v)-2).
\end{eqnarray}
Moreover,
a quasi-elliptic fibration
only admits fibres of Kodaira types 
\begin{eqnarray}
\label{eq:qe}
II, IV,
IV^*, II^*,
\end{eqnarray}
but the latter two do not appear on fibrations induced by lines on  smooth quartics.
A comparison of \eqref{eq:eqe} with Table 1 
yields that $\pi$ accommodates exactly 30 lines arranged in 10 stars (type $IV$).
This completes the proof of the lemma.
\end{proof}

\begin{Remark}
We do not expect the bound of 24 lines in the elliptic fibration case to be sharp (just like in \cite{RS}),
but possible improvements are not relevant to the goals of the present work.
\end{Remark}

\subsection*{Proof of Theorem \ref{thm}}

Given a smooth quartic surface $S\subset\PP^3$,
assume that there is some hyperplane $H$ which splits into 4 lines on $S$:
\[
H = \ell_1+\ell_2+\ell_3+\ell_4.
\]
Any other line $\ell'\subset S$ satisfies $\ell'.H=1$, 
so it meets one of the $\ell_i$.
By Lemma \ref{lem:30}, any line meets at most 30 other lines on $S$,
so we obtain
\begin{eqnarray}
\label{eq:112}
\#\{\text{lines on } S\} \leq 4 + 4\cdot (30-3) = 112.
\end{eqnarray}
Next assume that there is no such hyperplane on $S$.
This implies that any line $\ell$ on $S$ meets at most $12$ other lines:
by assumption, any line meeting $\ell$ is part of a unique fibre of Kodaira type $I_2$ or $III$
of the  fibration $\pi$ induced by $\ell$;
since this fibration is elliptic by inspection of the fibre types (compare \eqref{eq:qe}),
equation
\eqref{eq:ee} combined with Table \ref{tab:F} gives the claim.

If there is some hyperplane $H$ splitting off two lines on $S$,
\[
H = \ell_1 + \ell_2 + Q
\]
for some irreducible conic $Q$,
then we find as above 
\[
\#\{\text{lines on } S\} \leq 2 + 2\cdot (12-1) + \#\{\text{lines $\neq \ell_1, \ell_2$ on  $S$ meeting } Q\}\leq 70.
\]
(This improves to an upper bound of $68$ if we take into account in \eqref{eq:48} that $\ell_1.Q=\ell_2.Q=2$, compare the proof of Proposition \ref{prop:2} in Section \ref{s:2}.)

Finally, if all lines on $S$ are skew,
then there cannot be more than 21 of them
since they give mutually orthogonal $(-2)$-classes in $\Pic(S)$,
a lattice of signature $(1,\rho-1)$ with $\rho\leq b_2(S)=22$. 
This completes the proof of Theorem \ref{thm}.
\qed

\begin{Remark}
After this paper was completed,
Noam Elkies pointed out to us
that is might be possible to obtain an alternative proof of Theorem \ref{thm}
from bounds on spherical codes (see \cite{CK}).
Potentially this also applies to the uniqueness proved in the next section.
\end{Remark}

\section{Uniqueness}
\label{s:unique}

We continue to work over an algebraically closed field $k$ of characteristic $3$.
It is immediate that a smooth quartic $S\subset\PP_k^3$ can only contain 112 lines if it 
contains a line inducing a quasi-elliptic fibration.
In fact, a quick inspection of the proof of Theorem \ref{thm} reveals much more than this:
\begin{Fact} 
\label{fact}
On a smooth quartic surface, 
112 lines can only be attained 
if every line induces a quasi-elliptic fibration.
\end{Fact}

In this section, we will prove that a smooth quartic $S\subset\PP^3_k$ containing 112 lines
is unique up to the action of $\mbox{PGL}(4,k)$.
As a warm-up,
we will sketch how uniqueness can be proved as a K3 surface (disregarding the polarisation)
and as a quasi-elliptic K3 surface (in view of Fact \ref{fact}).

\subsection{Uniqueness as abstract K3 surface}

We know that $S$ contains a line $\ell$ inducing a quasi-elliptic fibration
\begin{eqnarray}
\label{eq:pi}
\pi: S \to \PP^1
\end{eqnarray}
necessarily with 10 reducible fibres of Kodaira type $IV$.
This already implies that $S$ is supersingular, i.e. $\rho(S)=b_S(S)=22$,
in agreement with the fact that $S$ is unirational.
All lines in $S$ disjoint from $\ell$ serve as sections for $\pi$.
The Mordell-Weil group is a finite $3$-elementary abelian group,
i.e.
\[
\MW(S,\pi) \cong (\Z/3\Z)^e \;\;\; 0\leq e\leq 4.
\]
In order to reach the total number of 112 lines,
we need $112-1-30=81$ sections, so $e=4$.
Then the discriminant formula for the N\'eron-Severi lattice of $S$ from \cite[(22)]{SSh} returns
\[
\disc(\NS(S)) = \disc(U+A_2^{10})/(\#\MW(S, \pi))^2 = -3^2.
\]
In other words, $S$ is a supersingular K3 surface of Artin invariant $\sigma=1$.
By Ogus' work \cite{Ogus}, $S$ is unique up to isomorphism.

\subsection{Uniqueness as quasi-elliptic surface}

By Example \ref{ex:Fermat} (compare  \cite{S-sigma=1} for the general case),
$S$ admits a model over $\F_9$ with generators of $\NS(S)$ defined over the same field.
It follows that the same holds true for any elliptic or quasi-elliptic fibration on $S$ 
since any such is given by the linear system associated to a divisor of Kodaira type 
(compare the argument in \cite[\S 9]{ES-2} and the classification in \cite{Sengupta}).
In particular, the reducible fibres of the quasi-elliptic fibration \eqref{eq:pi} (ten times type $IV$)
sit exactly over the points of $\PP^1(\F_9)$.
We compare this with the explicit shape that a quasi-elliptic K3 surface takes:
\begin{eqnarray}
\label{eq:WF}
y^2 = x^3 + f(t),\;\;\; f\in k[t], \; \deg(f)=12.
\end{eqnarray}
By the Jacobi-criterion, the singularities of the Weierstrass model sit exactly in the fibres at the roots
of the formal derivative $f'(t)$. 
Thus we can assume $f'(t)=c(t^9-t)$
which fixes $f$ up to cube powers of $t$ and a scalar multiple.
The former can be eliminated by a translation in $x$,
the latter by rescaling $(x,y)$ over $\bar\F_9$.
In consequence,
$S$ is given as a quasi-elliptic K3 surface
by the Weierstrass form
\[
S:\;\;\; y^2 = x^3 + t^{10} + t^2.
\]

\subsection{Uniqueness as  quartic surface}

In this section, we will finally prove the uniqueness of $S$ as a polarised K3 surface,
i.e. as a  quartic in $\PP^3$.
As before, we phrase our results over an algebraically closed field $k$ of characteristic $3$:

\begin{Proposition}
\label{prop:4}
Up to the action of $\mbox{PGL}(4)$
there is a unique smooth quartic surface in $\PP^3$ containing 112 lines.
\end{Proposition}

Before starting the proof, we record a simple, but crucial observation:

\begin{Lemma}
\label{lem:triple}
Let $\ell\subset S$ with induced quasi-elliptic fibration $\pi: S\to \PP^1$.
Then $\ell$ meets every fibre of $\pi$ with multiplicity three at its singular point.
\end{Lemma}

\begin{proof}
The intersection pattern of the lemma
is the only possibility for $\ell$ to meet every fibre in a single point.
Assuming that this is not the case,
 $\ell$ induces a morphism
\[
\ell\to\PP^1
\]
of degree $2$ or $3$ and thus with at most 4 ramification points.
Therefore it suffices to prove the claim for reducible fibres since there are 10 in number.

Assume that $\ell$ intersects the components of a $IV$ fibre
in three distinct points.
Pick any component $\ell'$ and consider the induced fibration $\pi'$.
Then $\ell$ and the other two components form a triangle,
i.e. a fibre of Kodaira type $I_3$ for $\pi'$.
By \eqref{eq:qe}, $\pi'$ cannot be quasi-elliptic, contradicting Fact \ref{fact}.
\end{proof}

\begin{proof}[Proof of Proposition \ref{prop:4}]
After a linear transformation, we can assume that the smooth quartic $S$
contains the line 
\[
\ell=\{x_3=x_4=0\}\subset\PP^3.
\]
Moreover we can ensure in the same way that two of the reducible fibres (type $IV$, cf.~Fact \ref{fact})
are located at $0, \infty$ (i.e.~$x_3=0$ resp.~$x_4=0$),
and that some fibre component is given by $x_1=0$ resp.~$x_2=0$.
Then by Lemma \ref{lem:triple} all fibre components intersect at $[0,1,0,0]$ resp.~$[1,0,0,0]$.
This limits the defining polynomial $f$ of $S$ to the following shape
\[
f = x_3 x_2 (x_2-x_3)(x_2+ax_3) + x_3x_4g(x_1,x_2,x_3,x_4) + x_4x_1(x_1-x_4)(x_1+bx_4)
\]
where $g$ is homogeneous of degree $2$ and we have rescaled coordinates to normalise two further coefficients.
One verifies directly that $\ell$ intersects the residual cubic $F_t$
in the fibre at $x_4=tx_3$ only at the point given by
\begin{eqnarray}
\label{eq:pt}
x_2^3+tx_1^3=0.
\end{eqnarray}
By Lemma \ref{lem:triple}
this ought to define a singular point on every fibre.
We compare this with the partial derivates of the residual cubic $h_t$ defining $F_t\subset\PP^2_{k(t)}$ 
which we restrict to $\ell$:
\[
\frac{\partial h_t}{\partial x_1}\mid_\ell \equiv \frac{\partial h_t}{\partial x_2}\mid_\ell \equiv 0,\;\;\; 
\frac{\partial h_t}{\partial x_3}\mid_\ell = (a-1)x_2^2+tg(x_1,x_2,0,0)+(b-1)t^2x_1^2.
\]
For degree reasons, the last partial derivative may only vanish at the intersection point from \eqref{eq:pt} for all $t$
if it vanishes identically:
\begin{eqnarray}
\label{eq1}
a=b=1, \;\; g\in(x_3,x_4).
\end{eqnarray}
We continue to argue analogously with the lines 
\[
\ell_1=\{x_1=x_3=0\},\;\;\; \ell_2=\{x_2=x_3=0\}
\]
and the induced genus one fibrations. The line $\ell_1$ meets the residual cubics in a single point
(given by $x_2^3+(t+c)x_4^3=0$ for some constant $c\in k$) if and only if
\begin{eqnarray}
\label{eq2}
g\in (x_1,x_3)+(x_4^2).
\end{eqnarray}
The analogous condition for $\ell_2$ leads to
\begin{eqnarray}
\label{eq3}
g\in (x_2,x_4)+(x_3^2).
\end{eqnarray}
In total, \eqref{eq1} - \eqref{eq3} limit the polynomial $g$ to the following shape:
\[
g=c_1x_1x_4+c_2x_2x_3+c_3x_3^2+c_4x_4^2+c_5x_3x_4,\;\;\; c_i\in k.
\]
Solving for the threefold intersection point to be a singularity of the residual cubic directly gives
\[
c_1=c_2=c_5=0.
\]
That is, $S$ is given by the polynomial
\[
f=x_3 x_2 (x_2^2-x_3^2) + x_3x_4(c_3x_3^2+c_4x_4^2) + x_4x_1(x_1^2-x_4^2).
\]
As the final twist, the above polynomial's shape is preserved by the translations
\[
x_1 \mapsto x_1+\alpha x_3, \;\; x_2\mapsto x_2+\beta x_4.
\]
These can be used to eliminate the coefficients $c_3, c_4$ from $f$ as well.
In conclusion, we have applied suitable linear transformations to reach a unique defining quartic polynomial for $S$.
\end{proof}

\begin{Remark}
One easily verifies that the resulting quartic is indeed isomorphic to the Fermat quartic from Example \ref{ex:Fermat} 
(over $\F_9)$, or to the reduction of a suitable model of Schur's quartic \cite{schur} over $\Q(\sqrt{-3})$.
\end{Remark}

\section{Characteristic 2}
\label{s:2}

This section sketches a proof of Proposition \ref{prop:2}.
We work over an algebraically closed field $k$ of characteristic $2$.
Since we do not expect the bound to be sharp (but also lack good examples),
we will give only the rough ideas for brevity.
Overall, one proceeds as for characteristic $3$ in Section \ref{s:pf}.

\begin{Lemma}
\label{lem:20}
A line on a smooth quartic surface $S\subset\PP^3_k$ meets at most
20 other lines on $S$.
\end{Lemma}

\begin{proof}[Idea of proof]
Let $\ell$ denote the line and $\pi: S\to\PP^1$ the induced genus one fibration.
If $\pi$ is elliptic,
then  the arguments in \S 3, 4 and Lemma 5.2 of \cite{RS}, 
leading to Corollary 5.3 in loc.~cit.~with the corresponding bound for characteristics $\neq 2,3$, can be applied 
after some mild modifications. (For instance, the Hessian has to be manipulated to work modulo $2$.)

If $\pi$ is quasi-elliptic,
then it may a priori admit fibres of Kodaira types $II, III,$ 
$I_n^*, III^*, II^*$.
Presently the latter three are ruled out by the smoothness of $S$.
Hence by \eqref{eq:eqe} there will be exactly 20 fibres of type $III$,
each containing a line meeting $\ell$.
\end{proof}

\subsection*{Proof of Proposition \ref{prop:2}}

The proof closely follows the line of the arguments of Theorem \ref{thm} in Section \ref{s:pf}.
If there is some hyperplane $H$ which splits into 4 lines on $S$,
\[
H = \ell_1+\ell_2+\ell_3+\ell_4,
\]
then Lemma \ref{lem:20} gives 
\[
\#\{\text{lines on } S\} \leq 4 + 4\cdot (20-3) = 72.
\]
If there is some hyperplane $H$ splitting into two lines and an irreducible conic  on $S$,
\[
H = \ell_1 + \ell_2 + Q
\]
then we find 
\[
\#\{\text{lines on } S\} \leq 2 + 2\cdot (20-1) + \#\{\text{lines $\neq \ell_1, \ell_2$ on  $S$ meeting } Q\}\leq 84.
\]
Here we improve on the bound for $Q$ from Theorem \ref{thm:Q} from 48 lines to 46 lines
reasoning with the surface $Z$ from Proposition \ref{prop:Z} as follows.
Consider the divisor $D=Z\cap S\in\OO_S(16)$.
If $m$ denotes the multiplicity of $Q$ in $D$,
then $\tilde D = D-mQ-\ell_1-\ell_2$ is effective without containing $Q$ in its support.
Hence
\begin{eqnarray}
\label{eq:48}
\#\{\text{lines $\neq \ell_1, \ell_2$ on $S$ meeting } Q\} \leq
\begin{cases}
\deg(\tilde D) = 62-2m,\\
Q.\tilde D = 28+2m.
\end{cases}
\end{eqnarray}
Comparing the two bounds yields $ \#\{\text{lines $\neq \ell_1, \ell_2$ on  $S$ meeting } Q\}\leq 44$.

Finally, if all lines on $S$ are skew,
then there are at most 21 of them as before.
This completes (the sketch of) the proof of Proposition \ref{prop:2}.
\qed

\subsection*{Acknowledgements}

We are grateful to Noam D.~Elkies and the anonymous referee 
for their valuable comments.

\end{document}